\documentclass[10pt,a4paper,reqno]{amsart}
\usepackage{amsfonts,amsthm,latexsym,amsmath,amssymb,amscd,amsmath, epsf}
\usepackage{graphicx}

\newtheorem{definition}{Definition}
\newtheorem{theorem}{Theorem}
\newtheorem{lemma}{Lemma}
\newtheorem{corollary}{Corollary}
\newtheorem{proposition}{Proposition}

\newtheorem{problem} {Problem}

\newtheorem{conjecture}{Conjecture}

\newcommand{\al}{\alpha}

\newcommand{\la}{\lambda}

\newcommand {\bC} {\mathbb {C}}

\newcommand {\bCP} {\mathbb {CP}^1}

\newcommand {\Ga} {\Gamma}

\newcommand {\eps} {\epsilon}

\newcommand {\HH} {\mathcal H}
\newcommand {\CC} {\mathcal C} 
\newcommand {\ST} {\mathcal {ST}} 
\newcommand {\SD} {\mathcal {SD}} 
\newcommand {\AST} {\mathcal {AST}}

\begin{document}
             \numberwithin{equation}{section}

             \title[A note on Evgrafov-Fedoyuk's  theory and quadratic differentials]
             {A note on Evgrafov-Fedoyuk's  theory and quadratic differentials}

\author[B.~Shapiro]{Boris Shapiro}
\address{Department of Mathematics, Stockholm University, SE-106 91, Stockholm,
            Sweden}
\email{shapiro@math.su.se}
\date{\today}

\begin{abstract} The purpose of this short paper is to recall the theory of  the (homogenized) spectral problem for a Schr\"odinger equation with a polynomial potential developed in the 60's by M.~Evgrafov with M.~Fedoryuk, and, by Y.~Sibuya  and its relation with  quadratic differentials.  We derive from these results  that the  accumulation rays of the eigenvalues of this problem are in $1-1$-correspondence with the short geodesics of the singular planar metrics on $\bC P^1$ induced by the corresponding quadratic differential. Using this interpretation we show that for a polynomial potential of degree $d$ the number of   such accumulation rays can be any positive integer  between $(d-1)$ and $d \choose 2$. 

\end{abstract}

\maketitle


\section{Introduction}
\label{sec:int}

This paper is motivated by a recent progress in the study of the spectral discriminant in the space of polynomial potentials, see   \cite  {GE},  \cite{GE1}, \cite{AG}. We start by explaining some important results scattered through  \cite{Fe1}, \cite{EF},   \S 6 in ch. III of \cite {Fe}, and \cite{Sib1}. 

Consider  a differential equation of the form 
\begin{equation}Ê\label{eq:Schr}
-y''+P(z)y=0
\end{equation} where $P(z)=a_0z^d+a_1z^{d-1}+...+a_d,\; a_0\neq 0$ is  an arbitrary  complex polynomial potential  of degree $d$.   
Setting $\phi_0=\arg a_0$  we define the (open) {\bf Stokes sectors} $S_j, j=0,...,d+1$ of \eqref{eq:Schr}  as  given by the condition: 
$$S_j=\left\{z: \left|\arg z-\frac {\phi_0}{d}- \frac{2\pi j}{d+2}\right|< \frac{\pi}{d+2}\right\}.$$
We consider the Stokes sectors $S_j$ as cyclicly ordered on $\bC$, i.e. $S_{d+2}$ is neighboring to $S_{d+1}$ and $S_0$. Notice that for an arbitrary   $P(z)$ the definition of its Stokes sectors depends only on $\phi_0$. In particular, the multiplication of $P(z)$ by a positive number preserves   the set of the Stokes sectors.  

It is well-known, see e.g. \cite{Sib}, \cite{HsSi} that for each open sector $S_j$ there exists and unique (up to a scalar factor) non-trivial solution $Y_j$ of   \eqref{eq:Schr} which is exponentially decreasing along any ray within this sector. Such a $Y_j$ is called {\bf subdominant}Ê in $S_j$.  It is convenient to think of $Y_j$'s as points on $\bCP$ where the latter space is the projectivization of the linear space of solutions of \eqref{eq:Schr}.)  For a generic  $P(z)$ all its $Y_j,\,j=0,...,d+1$ are distinct. Moreover, all  restrictions on possible configurations of subdominant solutions considered as a divisor of degree $d+2$ on $\bCP$ were described  by R.~Nevanlinna already in  \cite {Nev}. These restrictions are: (i) one has at least $3$ geometrically distinct points in this divisor and (ii) the multiplicity of each point of this divisor does not exceed $\left[\frac{d+2}{2}\right]$.  These restrictions  have a  strong and so far still unexplained resemblance with the notion of stability in  modern algebraic geometry. 

The main object of study of many papers is the following important complex analytic hypersurface and its restrictions to different natural subspaces. 
\begin{definition}\label{spec} The (extended) spectral discriminant $Spc_d\subset \CC_d$ is the set of all potentials $P(z)$ of degree $d$ such  that   \eqref{eq:Schr}  possesses a solution subdominant in (at least) $2$ different Stokes sectors. Here $\CC_d\simeq \bC^*\times \bC^d$ stands for the set of all coefficients $(a_0,a_1,...,a_d),\; a_0\neq 0$.
\end{definition}Ê

For any fixed $a_0$ the restriction of $Spc_d$ is given by an entire function of the variables $(a_1,...,a_d)$ and the description of the behavior of the spectrum when $a_0\to 0$ can be found in e.g. \cite {GE1} and references therein.  The group of affine coordinate changes $z\to az+b$ acts on the space of equations \eqref{eq:Schr} preserving $Spc_d$. Finally, notice that the space $\CC_d$ carries the obvious free action of $\bC^*$ by multiplication of a arbitrary polynomial of degree $d$ by a non-vanishing complex number. This action does not preserve $Spc_d$ and the main goal of this paper is to 
present some new results about the restriction of $Spc_d$ to the orbits of the latter action.

In other words, we consider the following (homogenized) spectral problem.

\begin{problem}
 For a given  polynomial $P(z)$ of degree $d$  find the set of all non-vanishing $\la\in \bC^*$   for which the equation 
\begin{equation}\label{eq:anal}
      -y''+\la^2 P(z)y=0, 
      \end{equation}
Ê has a solution subdominant in at least two distinct Stokes sectors.
\end{problem}

This problem was consider in the writings of M.~Fedoryuk (some of them joint) with M.~Evgrafov and  as well as Y.~Sibuya, F.~W.~J.~Olver and, in fact, even much earlier in the Ph.D. thesis of G.~D.~Birkhoff  from 1914, see \cite{Bi}. The spectrum of this problem is always discrete. The main known results about the latter problem are
\begin{enumerate}
 \item[(*)] the description of the condition determining  the accumulation rays of the spectrum, see Theorem 6.2, \cite {Fe1}; 
 
 \item[(**)] the description of the asymptotic density of the eigenvalues along such an accumulation ray,  see Theorem 6.3, \cite {Fe1}.
 \end{enumerate}

To present these results in details  we need to recall  the notion of the {\em Stokes graph} and of  {\em Stokes complexes}  of \eqref{eq:Schr},  see e.g. \cite {EF} and \cite {Wa2}.  (We follow the terminology of M.~Fedoryuk which apparently is not quite standard in the field.) 

\medskip 
\noindent 
{\em Notation.} Each root of $P(z)$ is classically called a {\em turning point} of (\ref{eq:Schr}). 
 A {\em Stokes line} of (\ref{eq:Schr}) is a containing at most two turning points 
 (finite or infinite) segment of the  real 
 analytic curve solving (w.r.t. $z$) the equation:  
 \begin{equation} \label{eq:stokes}
 Re\, \xi_{z_{0}}(z)=0\quad \text{where} \quad \xi_{z_{0}}(z)=\int_{z_{0}}^z\sqrt{P(u)}du,
 \end{equation} 
   $z_{0}$ being one of the turning points of (\ref{eq:Schr}). 
 The {\em Stokes graph} (denoted by $\ST_{P}$) of the equation~(\ref{eq:Schr}) is 
 the union of all its Stokes curves. Connected components of $\ST_{P}$ will be 
 referred to as {\em Stokes complexes} and 
 connected components  of $\bC\setminus \ST_{P}$ will be called {\em 
 admissible domains}. 
  One can distinguish between two natural types of admissible 
     domains. Namely, an admissible domain is called the {\em 
     half-plane type} if the function $\xi(z)$ maps it to 
     $Re\,\xi(z) >a$ or $Re\,\xi(z) <a$ for some real $a$. 
     Analogously, an admissible domain is called the {\em 
     strip type} if the function $\xi(z)$ maps it to 
     $a<Re\,\xi(z) <b$ for some real $a<b$. For polynomial $P(z)$ all 
     admissible domains belong to one of these two type (which is no 
     longer true for entire or rational potentials). 

 A Stokes complex is called {\em  simple} if it 
 contains exactly one turning point and {\em non-simple} otherwise. (In Russian, M.~Fedoryuk calls 'non-simple complexes'  as  {\em 'complex complexes'}Ê  but this  terminology sounds  rather odd in English.)   Note that the existence of a non-simple Stokes complex in the Stokes graph of $P(z)$ is equivalent  to the existence of a  Stokes line between two turning points. Such a Stokes line will be called {\em finite} or {\em short}. 
 
 Given a polynomial $P(z)$ consider the family $P_t(z)=e^{2t\sqrt{-1}}P(z),\;\; t\in [0,\pi]$. We will briefly discuss how the Stokes graph changes in this family. Theorem~6.1 of \cite {Fe1}Ê claims that for any $P(z)$ with at least two distinct roots there exist only finitely many values of $t$ for which $P_t(z)$ has a non-simple Stokes complex.  We call a polynomial $P(z)$ {\em Fedoryuk-generic} if it satisfies the following 3 conditions: 
 \begin{enumerate}
 \item[(a)] all its roots are simple;
 \item[(b)] all non-simple complexes arising in the family $P_t(z)$ contain exactly two turning points;
 \item[(c)] for any value of $t$ the Stokes graph of $P_t(z)$ has either none or exactly one non-simple Stokes complex.
 
 \end{enumerate}
 
 A slight generalization and reformulation of Theorem~6.2, \cite{Fe1} adjusted to our context Ê reads as follows. 

\begin{proposition}\label{pr:Fed} Given a Fedoryuk-generic potential $P(z)$ denote by $t_1,...,t_k$ the values of the parameter $t\in[0,\pi]$ for which the Stokes graph of $P_t(z)$ has (and unique) non-simple Stokes complex. Then for any $\eps>0$ all but finite many eigenvalues of the problem \eqref{eq:anal}Ê lie $\eps$-close to the union of $k$ rays with the slopes $\tan t_1,...,\tan t_k$. Moreover, near each such ray (called the {\em accumulation ray of \eqref{eq:anal}}) lie infinitely many such eigenvalues. 
      \end{proposition}
      
Generalizing Theorem~6.3, \cite{Fe1} we get   that the asymptotic density of the eigenvalues is described as follows. 

\begin{proposition}\label{pr:Feddens} In the notation of Proposition~\ref{pr:Fed} let  $t_j$ be one the values of the parameter $t\in[0,\pi]$ for which the Stokes graph of $P_t(z)$ has (a unique) non-simple Stokes complex. Fixing a sufficiently small $\eps>0$  let $\la_1^{(j)}, \la_2^{(j)}, ... \la_n^{(j)},...$ be the sequence of the eigenvalues of \eqref{eq:anal}Ê   (non-strictly) ordered by their absolute values and lying $\eps$-close to the $j$-th ray, i.e. to the one  with the slope $\tan t_j$. Then when $|\la|\to \infty$ one has the following asymptotic expansion: 
\begin{equation}  \label{eq:asymp}
\la_n \int_\CC \sqrt{P(\xi)} d \xi \sim 2\pi n +\pi +\sum_{j=1}^\infty \la_n^{-j}\int_\CC\al_j(\xi)d\xi. 
\end{equation} 
Here $\CC$ is a simple closed curve containing the short Stokes line and no other turning points in its interior.  
The family of functions $\al_j(z),\; j=0, 1,...$ are defined as follows 
\begin{equation}\label{eq:recc}
\al_0(z)=-\frac{p'(z)}{4p(z)},\;\; \al_{j}(z)=-\frac{1}{2\sqrt{p(z)}}\left(\sum_{m=0}^{j-1}\al_m(z)\al_{j-m-1}(z)+\al_{j-1}'(z)\right).
\end{equation}
\end{proposition}
      
          Let us now reformulate the above statements in terms of  quadratic differentials and present some new results. We start with standard definitions.  The basic references for quadratic differentials are \cite{Str} and \cite{Je}.

\begin{definition} A {\em meromorphic quadratic differential $\Psi$} on a (compact) Riemann  surface $\Ga$ is a meromorphic section of  the square $(T_\bC^*\Ga)^{\otimes 2}$ of the holomorphic cotangent bundle $T_\bC^*\Ga$.   The zeros and poles of $\Psi$ are called its {\em singular}  points.  The set of all singular points of $\Psi$ on $\Ga$ is denoted by $Sing_\Psi$.
\end{definition}

If $z$ is a local holomorphic coordinate near a point $p$ on $\Ga$ then any  meromorphic quadratic differential $\Psi$ is locally represented as $f(z)(dz)^2$ where $f(z)$ is a meromorphic function.

\begin{definition} Given a meromorphic quadratic differential $\Psi$ on $\Ga$ we define two distinct foliations on $\Ga\setminus Sing_\Psi$ as follows. At each non-singular point
there are exactly two directions at which $\Psi$ attains positive/resp. negative values. Integral curves of these direction fields are called {\em the horizontal resp. vertical trajectories} of  $\Psi$.
     \end{definition}

     Obviously, the direction fields  are orthogonal  at each non-singular point.  In the theory of quadratic differentials one usually forgets  about the vertical direction field and studies only the trajectories of the horizontal one. 

\begin{definition}[comp. Definition 20.1 in \cite {Str}]  A trajectory of $\Psi$ is called {\em singular}  if its starts or ends at a singular point.
\end{definition}

\begin{definition} The {\em canonical length element} on $\Ga$ associated with a  quadratic differential $\Psi$ given locally as $\Psi=f(z)(dz)^2$  is defined by
$$|dw|=|f(z)|^{\frac{1}{2}} |dz|.$$
(Local) minimizers of the latter length element are called the {\em geodesics} Êof the quadratic differential 
$\Psi$. Geodesics whose both  endpoints are singular points are called {\em short}.
\end{definition}





\begin{definition}
The {\em distinguished} or {\em canonical} parameter associated with a quadratic differential $\Psi=f(z)(dz)^2$ is defined as $$W=\int\sqrt{f(z)} dz$$
 for some branch of the square root.
\end{definition}

Notice that geodesics are (locally) straight lines in the canonical parameter and short Stokes lines in the family $P_t(z)$ connecting the turning points  are exactly the short  geodesics of the quadratic differential $P(z)dz^2$ associated with the hyperelliptic curve $y^2=P(z)$. Thus, the following 
result holds.

\begin{corollary} \label{cor:quaddiff}
For a Fedoryuk-generic polynomial $P(z)$ the set  of the accumulation rays of the problem \eqref{eq:anal} is in $1-1$-correspondence with the set of the short geodesics of  $P(z)dz^2$. 
\end{corollary}

The next result seems to be new. 

\begin{theorem}\label{th:estimate} For any polynomial $P(z)$ of degree $d$ the number  of  short geodesics of the quadratic differential $P(z)dz^2$ can take an arbitrary integer value in between $d-1$ and $d  \choose  2$.
\end{theorem}Ê

\medskip
\noindent 
          {\it Acknowledgements.}    I am deeply grateful  to Professors Y.~Baryshnikov,  A.~Eremenko, A.~Gabrielov, S.~Giller and A.~Zorich for important discussions around this topic. I want especially to thank my Ph.D. student T.~Holst for his patience and help with several details.

\section{Proofs}
\label{sec:asymp}

\begin{definition} Let $\Psi$ be a quadratic differential on $\bC P^1$. A $\Psi$-polygon is
a simple closed curve consisting of a finite number of (possibly singular)
geodesics of finite length.
\end{definition}Ê

Let $\Gamma$ be a $\Psi$-polygon with an interior domain $D$ and assume that $z_j$ are
the singular points on $\Psi$ of orders $n_j$ and that $\xi_i$ are the singular points
inside $D$ of orders $n_i$. Moreover, let the interior angles at the vertices of $\Gamma$ 
be denoted by $\theta_j,  0 \le  \theta_j \le  2\pi$. The following  result holds, see e.g. Theorem 14.1 of \cite{Str}.

\begin{theorem} [Teichm\"uller's lemma]  One has
\begin{equation}\label{eq:Teich}
\sum_j\left(1-\frac{(n_j+2)\theta_j}{2\pi}\right)=2+\sum_in_i.
\end{equation}
\end{theorem}

Consider now a quadratic differential $\Psi=P(z)dz^2$ on $\bC P^1$ where
$P(z)$ is a polynomial of degree $d$. We want to count the total number of unbroken
geodesics connecting pairs of roots of $P(z)$.  (By 'an unbroken' geodesic we mean a geodesic not passing through other roots of $P(z)$ except the two chosen.)

\begin{lemma}\label{lm:2roots}
For any given pair of roots of $P(z)$ there exists at most $1$ unbroken geodesic connecting them.\end{lemma} 

\begin{proof} Follows straightforwardly from Teichm\"uller's lemma. Indeed, if there were two such geodesics then they will form a $\Psi$-polygon $\Gamma$ splitting  $\bC P^1$ into two connected domains. Let $D$ be the "bounded domain", i.e. $D$ is the component of $\bC P^1\setminus \Gamma$  not containing $\infty\in \bC P^1$. By assumption there is exactly two singular points on $\Gamma$ and so the left-hand side of \eqref{eq:Teich}  is smaller than  $2$, whereas the
right-hand side is clearly at least $2$, a contradiction. 
\end{proof} 

This lemma immediately implies the inequalities in Theorem~\ref{th:estimate}. Namely, 

\begin{corollary}Ê
A arbitrary quadratic differential $\Psi=P(z)dz^2,\; \deg P(z)=d$  has at least $d-1$ and most $\binom d 2$ unbroken short geodesics.
\end{corollary}

\begin{proof}  The upper bound $d\choose  2$ is provided by the latter lemma. Moreover, 
it is clear that the  set formed by all short unbroken geodesics is compact and connected. Indeed,  any two roots of $P(z)$ are connected by  at least one  geodesic (broken or unbroken). Then assuming that all $d$ roots of $P(z)$ are distinct one needs  at least $d-1$ unbroken geodesic to guarantee that this set is connected, which can be realized for instance by assuming that  $P(z)$ has only real  roots. (See more examples below.)
\end{proof} 

Let us now show that there exist polynomials $P(z)$ having an arbitrary number of unbroken geodesics between $d-1$ and $d \choose 2$. We will use the interpretation of the quadratic differential $P(z)dz^2$ where $P(z)$ is a monic polynomial of degree $d$ with the vanishing coefficient at $z^{d-1}$ (i.e. with the vanishing sum its roots) in terms of a pair of weighted chord diagrams on $d+2$ vertices, see \S~4 of \cite{Bar}.  

The main construction is as follows. Take $\Psi=P(z)dz^2$ as above and take its Stokes graph $\ST_P$ and anti-Stokes graph $\AST_P$. (The anti-Stokes graph of $P(z)dz^2$ is by definition the Stokes graph of $-P(z)dz^2$.) Unbounded Stokes lines of $\ST_P$ tend 
at $\infty$ to $d+2$ standard  directions (called the Stokes rays) and analogously unbounded anti-Stokes lines of $\AST_P$ tend at $\infty$ to $d+2$ standard  directions (called the anti-Stokes rays). The set of all Stokes  and resp. anti-Stokes rays naturally form the set of vertices of two regular $(d+2)$-gons which are rotated w.r.t. each other by the angle $\frac{\pi}{d+2}$.  As we mentioned above admissible domains, i.e. connected components of $\bC P^1\setminus \ST_P$ (resp. $\bC P^1\setminus \ST_P$) are of two types: the half-plane type and the strip-type. The half-plane type components are mapped by the function $\xi(z)$, see {eq:stokes} into half-planes and the strip type to strips. For $\bC P^1\setminus \ST_P$ these half-planes and strips are bounded by the vertical lines and for $\bC P^1\setminus \AST_P$ by the horizontal lines.  Each strip-type domain is topologically an infinite strip bounded by two curves. Moreover,  the two  pairs of ends of its boundary approach two distinct and non-neighboring Stokes rays. Thus, one obtains that any strip-type domain connects two vertices of the Stokes  $(d+2)$-gon and  one can represent this strip-type domain by the corresponding chord in the Stokes $(d+2)$-gon. Analogously, strip-type domains for the anti-Stokes graph connect pairs of vertices  of the anti-Stokes $(d+2)$-gon. Moreover, we can assign to each strip-type domain a positive weight coinciding by definition  with the width of its image under the map $\xi(z)$. 
(Those weights are the real and imaginary parts of the integrals $\int \sqrt{P(t)}dt$ taken over certain paths connecting pairs of roots of $P(z)$ and are closely related to the periods of $y^2=P(z)$.) For a generic polynomial $P(z)$ the number of its strip-type domains for $\ST_P$ and $\AST_P$ equals $d-1$ which means that one gets weighted triangulations of both the Stokes and anti-Stokes polygons. In other words, one obtains a pair of weighted chord diagrams of the Stokes and anti-Stokes $(d+2)$-gons. The following statement can be found in \S 4 of \cite{Bar}. 

\begin{proposition}Ê\label{pr:bar}
The above procedure gives a $1-1$-correspondence between the set of all quadratic differentials of the form $P(z)dz^2$ with $P(z)=z^d+a_1z^{d-2}+...+a_{d-1}$ and 
the set of ordered pairs of weighted chord diagrams.  
\end{proposition}Ê

To find quadratic differentials $P(z)dz^2$ with a given number $k$ of short geodesics satisfying the inequality $d-1\le k \le {d \choose 2}$ we will only use a special class of quadratic differentials.  

\begin{definition}ÊA quadratic differential $P(z)dz^2,\; \deg P(z)=d$ is called {\em very flat}Ê if a) all roots of $P(z)$ are simple; \quad b) the number of its horizontal strip-type domains equals $d-1$; \quad c) each root of $P(z)$ lies in the closure of at most $2$ strip-type domains. 
\end{definition} 
 
Notice that condition c) above is equivalent to the following property of its Stokes chord diagram. Each domain obtained by removing the diagram from the Stokes $(d+2)$-gon has at most $2$ neighbors, see below. 

To each very flat quadratic differential we can naturally associate the following object whose versions earlier appeared in a number of books discussing the WKB-problems, see e.g. \cite {Sib}, p. 269. 

\begin{definition}Ê{\em A chopped vertical strip} is a set of complex numbers $z_1, z_2, ...., z_d$ where $z_j=x_j+\sqrt{-1}y_j,\; j=1,...,d$ with distinct real and imaginary parts and ordered by their real parts, i.e. $x_1<x_2<...<x_d$ together with a vertical ray going up or down at each $z_j,\; j=2,...,d-1$, see Fig.~1.   These rays are called the 
{\em cuts} of the strip. 
\end{definition}

\begin{figure}

\begin{center}
\includegraphics[scale=0.20]{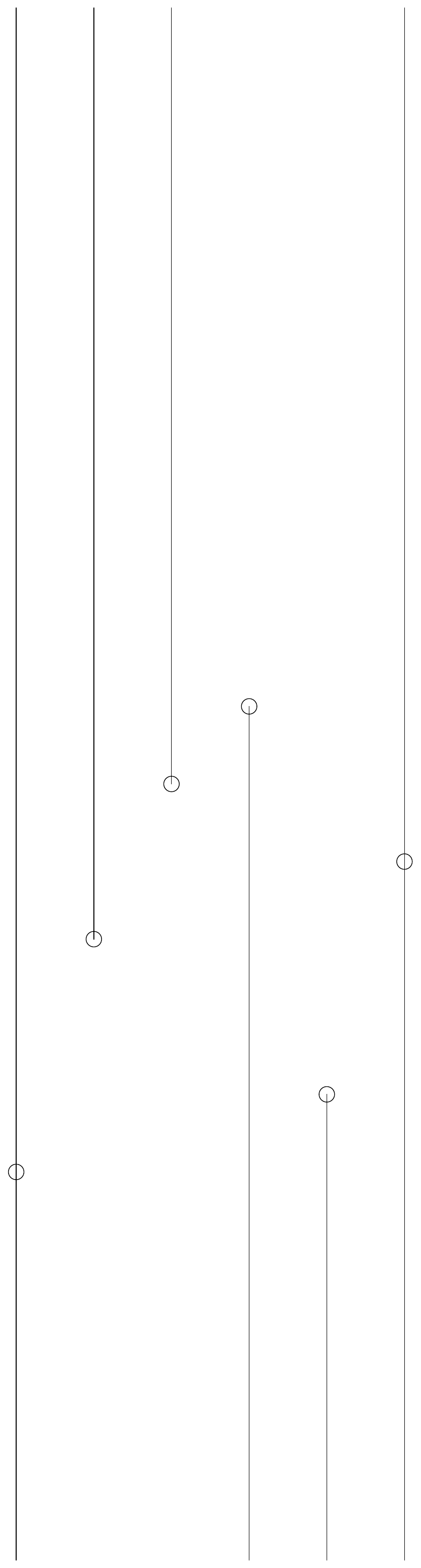}  \quad \quad  \quad  \quad  \quad\includegraphics[scale=0.20]{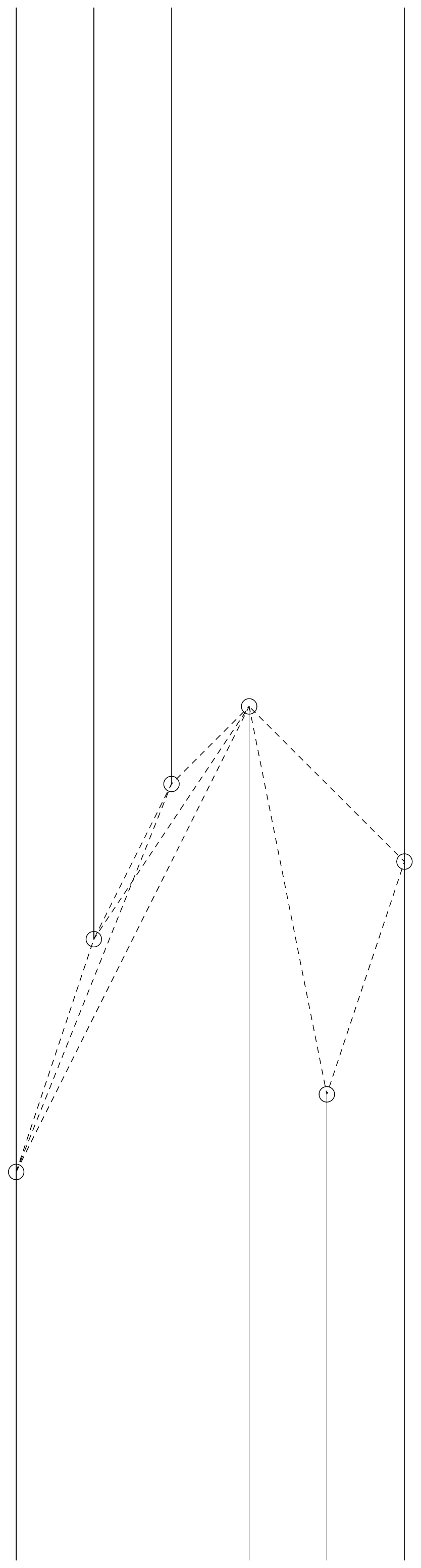} 
\end{center}

\caption{A chopped strip and the same strip with short geodesics.}
\label{fig1}
\end{figure}

The map associating to a very flat quadratic differential its chopped vertical strip is the same map $\xi(z)$ which we introduced earlier whose domain we restrict to 
$\bC P^1\setminus \bigcup_{j=1}^{d+2} O_j$ where $O_j$ is the closed $j$-th half-plane type domain. (It is usually convenient to assign the integration base point to the inverse image of $z_1$ which implies  $z_1=0$.)  

\begin{lemma}\label{lm:strip}
For any chopped strip there exists a (non-unique) very flat quadratic differential $P(z)dz^2$ mapped to this strip. Its short geodesics are in $1-1$-correspondence with straight segments connecting pairs of vertices of the strip and not intersecting its cuts. 
\end{lemma}

\begin{proof}Ê
Just throw away all half-plane domains and use $\xi(z)$. 
\end{proof}Ê

The following statement finishes the proof of Theorem~\ref{th:estimate}. 

\begin{proposition}\label{pr:realiz} 
There exist chopped strips with $d-2$ cuts and an arbitrary number $k$ of short geodesics satisfying  Êthe inequality $d-1\le k \le {d \choose 2}$.
\end{proposition} 

\begin{proof} We will use induction on $d$. Assume that one can realize any number $k'$ of short geodesics satisfying $d-2 \le k'\le {d-1 \choose 2}$ for $d-1$. 
Then it is very easy to construct chopped strips with $d$ nodes having exactly $k'+1$ short geodesics. Indeed, draw a cut up or down through the last point  and place the $d$-th in such a way that no short geodesics from points except for $d-1$ will be possible. The only new geodesic connects the $(d-1)$st and the $d$-th point. Thus,  we need to realize the number of short geodesics between ${d-1  \choose 2}+2$ and ${d \choose 2}$. This is easily done by using a convex $d-1$-gon. We can cut it but an additional cut and kill an arbitrary number of short geodesics between $0$ and $\frac{(d-1)^2}{4}$ (geodesics from the left half to the right half). Thus we get what we need for $d\ge 5$.
\end{proof} 

\section{Final remarks} 

\noindent
{\bf 1.}Ê The next question is motivated  by Theorem 2 of \cite {Shin}. Namely, he gave necessary and sufficient conditions for the usual spectral problem  for the Schr\"odinger equation with a polynomial potential  to have infinitely many real eigenvalues (or, similarly, eigenvalues on a given accumulation ray). His necessary and sufficient condition under appropriate assumptions is that there exist a point $z_0\in \bC$ such that  $P(z-z_0)$ is PT-symmetric. (See details in \cite{Shin}.) All tools he is using for the standard spectral problem have their natural counterparts for the problem~\eqref{eq:anal}. 

\begin{problem} Give necessary and sufficient conditions guaranteeing that the spectral problem ~\eqref{eq:anal} has infinitely many eigenvalues lying exactly on some accumulation ray. 
\end{problem}  

\noindent
{\bf 2.} ÊObviously,  not every polynomial $P(z)$ gives rise to a very flat potential $P(z)dz^2$, but in many cases one can obtain a very flat potential by multiplying $P(z)$ by a constant.

\begin{problem} Is it true that for any polynomial $P(z)$ there exists $t\in [0,2\pi)$ such that the differential $e^{t\sqrt{-1}}P(z)dz^2$ is very flat?  
\end{problem}

\end{document}